\documentclass[11pt,english]{article}
\usepackage[T1]{fontenc}
\usepackage[latin9]{inputenc}
\usepackage{babel}
\usepackage{amsmath, color}
\usepackage{amsthm}
\usepackage{amssymb}
\usepackage{breakurl}
\usepackage{latexsym}
\usepackage[all]{xy}
\usepackage{color}
\usepackage{enumerate,url}
\usepackage[pagewise]{lineno}

\makeatletter
\theoremstyle{plain}
\newtheorem{thm}{\protect\theoremname}
  \theoremstyle{plain}
  \newtheorem{lem}[thm]{\protect\lemmaname}
  \theoremstyle{remark}
  
  \theoremstyle{plain}
  \newtheorem{prop}[thm]{\protect\propositionname}
  \theoremstyle{definition}
  \newtheorem{definition}[thm]{\protect\definitionname}
  \newtheorem{example}[thm]{\protect\examplename}
  \theoremstyle{plain}
  \newtheorem{corollary}[thm]{\protect\corollaryname}

\usepackage{mathrsfs}
\usepackage{amscd}

\usepackage{graphics}
\usepackage{a4wide}

\usepackage{accents}

\usepackage{array}

\newcounter{EQNR}
\renewcommand{\contentsname}{Contents\\
{\footnotesize\normalfont(A table of contents should normally not
be included)}}

\usepackage{babel}
\providecommand{\lemmaname}{Lemma}
\providecommand{\theoremname}{Theorem}
\providecommand{\propositionname}{Proposition}
\providecommand{\remarkname}{Remark}
\providecommand{\examplename}{Example}
\providecommand{\definitionname}{Definition}

\makeatother

  \providecommand{\examplename}{Example}
  \providecommand{\lemmaname}{Lemma}
  \providecommand{\propositionname}{Proposition}
  \providecommand{\remarkname}{Remark}
\providecommand{\theoremname}{Theorem}
\providecommand{\corollaryname}{Corollary}

\begin{document}

\title{Constructing heat kernels on infinite graphs }


\author{Jay Jorgenson \footnote{The first-named author acknowledges grant support
from PSC-CUNY Award 65400-00-53, which was jointly funded
by the Professional Staff Congress and The City University of New York.} \and Anders Karlsson
\footnote{The second-named author acknowledges grant support by the Swiss NSF grants 200020-200400, 200021-212864 and the Swedish
Research Council grant 104651320.} \and Lejla Smajlovi\'{c}}
\maketitle

\begin{abstract}\noindent
Let $G$ be an infinite, edge- and vertex-weighted graph with certain reasonable restrictions.
We construct the heat kernel of the associated Laplacian using an adaptation of the parametrix approach due to Minakshisundaram-Pleijel in the setting of Riemannian
geometry. This is partly motivated by the wish to relate the heat kernels of a graph and a subgraph, and the wish to have an explicit expression of the heat kernel related to Gaussian-type estimates for any graph metric bounded from below.
Assuming uniform boundedness of the combinatorial vertex degree, we show that a dilated Gaussian depending on any distance metric on $G$, which is uniformly bounded from below can be taken as a parametrix in our construction.
We discuss several applications of parametrix construction.
For example, assuming that the graph is locally finite, we express the heat kernel $H_G(x,y;t)$ as a Taylor series with the lead term being $a(x,y)t^r$, where $r$ is the combinatorial distance between $x$ and $y$ and $a(x,y)$ depends (explicitly) upon edge and vertex weights.
In the case $G$ is the regular $(q+1)$-tree with $q\geq 1$, our construction reproves different explicit
formulas in papers by Chung-Yau, Cowling-Meda-Setti, and Chinta-Jorgenson-Karlsson.

\end{abstract}

\section{Introduction}
In the seminal article \cite{MP49}, Minakshisundaram and Pleijel established a general method by which
one can construct explicitly the heat kernel for the Laplacian operator associated to a smooth compact Riemannian
manifold.  As stated in \cite{MP49}, their method is a generalization of previous work by Carleman.  In effect,
one begins with an initial approximation for the heat kernel for time approaching zero, called a parametrix $H$,
and then one forms a Neumann series $F(H)$ of convolutions only involving the parametrix $H$.  It is then shown that
$H+F(H)$ equals the heat kernel sought.  So, in essence, the heat kernel is realized as a type of fixed point
theorem since one has some flexibility when choosing the parametrix $H$.

\subsection{Main results}

In this paper we will develop a similar methodology by which one can construct the heat kernel associated to
the graph Laplacian for a rather general infinite graph $X$.  Our approach is most directly
inspired by \cite{Mi49} and \cite{Mi53}
as well as \cite{Ro83} and \cite{Ch84}.  The results of this paper extend results from
\cite{CJKS23} in which a parametrix construction is carried out for finite graphs whose vertex weight
function is identically equal to one, compare also with \cite[section 5]{LNY21}. Note that the setting of infinite graphs with non-trivial vertex weights is more intricate and involve deeper analysis as well as an $L^p$- spaces approach to the construction. Moreover, our results complements those from \cite{Wo09} where the heat kernel on infinite graphs is constructed by exhausting the infinite graph with finite, connected subgraphs.

One motivation for undertaking this study is the following. Some of the most significant applications of heat kernel analysis (see \cite{JoLa01, Gr09}  and references therein) come from relating the heat kernels on two spaces $X$ and $Y$ via a quotient structure $Y\twoheadrightarrow X$, such as Poisson summation, Jacobi theta inversion and other trace formulas. The parametrix construction can naturally be used to instead set up a comparison in a subspace situation $X\hookrightarrow Y$. This is the point of view taken by Lin, Ngai and Yau in  \cite{LNY21}, who wrote that it is useful to express the heat kernel expansions of $X$, a finite graph, in terms of that of $Y$, a complete graph containing $X$. These authors viewed this in analogy with relating the heat kernel on a compact $d$-dimensional manifold with $\mathbb{R}^d$ as done in Riemannian geometry. In our setting, since our graphs $X$ are infinite, this would rather correspond to non-compact manifolds or open domains in $\mathbb{R}^d$. Note that we also allow our space $Y$, via the parametrix chosen, to be more general, for example it could be a graph containing $X$ or a domain that $X$ provides a discretization of. See also \cite{CY99} and \cite{HMW19} where a method of imaging is used to compare heat kernels on a graph and its cover.

The second motivation is to provide an \emph{explicit} expression for the heat kernel in terms of a rather general parametrix. This allows one to choose a parametrix that is suitable for applications. For example, the Dirac delta parametrix will yield to explicit formulas for the heat kernel on regular trees, while the choice of the dilated Gaussian as a parametrix (as in section \ref{sec: metric} below) can be used in deriving Gaussian estimates which, of course, will depend upon the chosen metric (adopted to the setting).

Let $G$ be an infinite connected graph with vertex weights $\theta(x)$ and edge weights $w_{xy}$ satisfying conditions (G1) and (G2) given below. The Laplace operator on functions is
\begin{equation*}
\Delta_{G}f(x)=\frac{1}{\theta(x)}\sum_{y\in VG}(f(x)-f(y))w_{xy}.
\end{equation*}
See section \ref{ref our setup} for all details. The main result of this paper is an explicit construction of the heat kernel on $G$ associated to this Laplacian, starting with a parametrix,
see Definition \ref{def. parametrix} for the definition:

\begin{thm}
\label{thm: formula for heat kernel} Let $H$ be a parametrix of
order $k\geq0$ for the heat operator $L_G=\Delta_{G}+\partial_t$.  For $x,y\in VG$ and $t\in \mathbb{R}_{\geq 0}$, let
\begin{equation*}
F(x,y;t):=\sum_{\ell=1}^{\infty}(-1)^{\ell}
(L_{G,x}H)^{\ast\ell}(x,y;t).
\end{equation*}
This series
converges absolutely and uniformly on every compact subset of $VG\times VG\times\mathbb{R}_{\geq 0}$.
The heat kernel $H_{G}$ 
is given by
\begin{equation*} 
H_{G}(x,y;t)=H(x,y;t)+(H\ast F)(x,y;t)
\end{equation*}
and
$$
(H\ast F)(x,y;t)
= O(t^{k+1})
\,\,\,\,\,
\text{\rm as $t \rightarrow 0$.}
$$
\end{thm}

Here, $\ast$ denotes the convolution of two functions with a domain $VG\times VG\times \mathbb{R}_{\geq 0}$, see Definition \ref{def. conv}. Our construction allows for a variety of expressions for the heat kernel, and we focus on two types of parametrix, the simple Dirac delta function (section \ref{sec. Dirac delta param}) and more elaborate parametrix coming from distance functions  (section \ref{sec. further ex}).

First, we show in section \ref{sec. Dirac delta param} that the Dirac delta function $\frac{1}{\theta(x)}\delta_{x=y}$ on $L^1(\theta)$ can be chosen as a parametrix,
under additional assumption that the graph is locally finite, meaning that for each $x\in VG$, the number of vertices adjacent to $x$ is finite.
Then, in Proposition \ref{prop. dirac delta HK} below we show that by taking the Dirac delta function as the parametrix in our main theorem, the heat kernel $H_{G}(x,y;t)$ on $G$ can be expressed as
\begin{equation}\label{eq. HK from dirac}
H_{G}(x,y;t)=\frac{1}{\theta(x)}\delta_{x=y}+ \sum_{\ell=1}^{\infty}(-1)^{\ell}\frac{t^{\ell}}{\ell !}\sum_{z_1,\ldots,z_{\ell-1} \in VG} \delta_x(z_1)\delta_{z_1}(z_2)\cdot \ldots \cdot \delta_{z_{\ell-1}}(y),
\end{equation}
where
\begin{equation}\label{eq:delta_definition}
\delta_x(y):=\frac{1}{\theta(x)}\left\{
                                                             \begin{array}{ll}
                                                               \mu(x), & x=y; \\
                                                               -w_{xy}, & x\sim y \\
                                                               0, & \text{otherwise,}
                                                             \end{array}
                                                           \right. \quad x,y \in VG.
\end{equation}

The combinatorial expression \eqref{eq. HK from dirac} for the heat kernel further highlights the local nature of the heat diffusion.  Specifically, the first $r+1$ terms in the expansion on the right-hand side of \eqref{eq. HK from dirac} in powers of $t$ depend only on points that are at combinatorial distance at most $r$ from the starting point $x$.
Moreover, if the combinatorial distance $d(x,y)$ between $x$ and $y$ is $r\geq 2$, then $\delta_x(z_1)\delta_{z_1}(z_2)\cdot \ldots \cdot \delta_{z_{\ell-1}}(y)=0$ for any choice $z_1,\ldots,z_{\ell-1} \in VG$ of $\ell-1\leq r-1$ points which yields the following corollary.

\begin{corollary} \label{cor. amall time asy}
Let $G$ be an infinite weighted, connected and locally finite graph satisfying conditions (G1) and (G2). Then, for all $x,y\in VG$ with the combinatorial distance $d(x,y)=r\geq 1$ and all $t>0$ we have
$$
\left|H_{G}(x,y;t)-(-1)^r\frac{t^r}{r!}c_r(x,y)\right|\leq t^{r+1} C(x,y;t),
$$
where $c_r(x,y)=\sum_{z_1,\ldots,z_{r-1} \in VG} \delta_x(z_1)\delta_{z_1}(z_2)\cdot \ldots \cdot \delta_{z_{r-1}}(y)$ and
$$
C(x,y;t)=\left| \sum_{\ell=r+1}^{\infty}(-1)^{\ell}\frac{t^{\ell - r-1}}{\ell !}\sum_{z_1,\ldots,z_{\ell-1} \in VG} \delta_x(z_1)\delta_{z_1}(z_2)\cdot \ldots \cdot \delta_{z_{\ell-1}}(y) \right|
$$
is bounded as $t\downarrow 0$.
\end{corollary}
Therefore, the heat kernel $H_{G}(x,y;t)$, for small time $t$ is asymptotically equal $ \frac{(-1)^r}{r!}c_r(x,y) t^r $. This was proved in \cite[Theorem 2.2.]{KLMST16}, with the constant multiplying $t^r$ and the upper bound $C(x,y;t)$ expressed in a different way.

Second, in section \ref{sec. further ex} we consider a parametrix which is inspired by the Gaussian as the heat kernel on $\mathbb{R}$ (see \cite{Gr09}, section 9 for generalizations and modifications in higher-dimensional setting).  Namely, in Proposition \ref{prop. Gaussian parametrix}
below we show that, assuming uniform boundedness of combinatorial vertex degree, for \emph{any} distance metric  $d$ on $G$ such that $d(x,y)\geq \delta>0$ for all distinct
points $x,y \in VG$, the function
$H_d:VG\times VG \times [0,\infty) \to [0,\infty) $, defined as
\begin{equation}\label{eq. defn param exp}
H_d(x,y;t):=\frac{1}{\sqrt{\theta(x)\theta(y)}} \exp(-(\theta(x)\theta(y)d^2(x,y))/t)
\,\,\,
\text{\rm with}
\,\,\,
H_d(x,y;0)=\lim_{t\downarrow 0} H_d(x,y;t)
\end{equation}
can be taken as a parametrix in our construction of the heat kernel.

\subsection{Applications}

An explicit expression for the heat kernel on an infinite graph associated to the bounded Laplacian in terms of a rather general parametrix can be applied to deduce a variety of results in different settings. In this section we will describe some of those.

\subsubsection{Explicit evaluations of heat kernels}

There has been extensive studies of heat kernels
within the field of spectral graph theory and geometry. However, as noted in \cite{LNY21}, there are very few instances of explicit evaluations of heat kernels on infinite graphs.  
The main examples are lattice graphs (see \cite{CY97, Be03, KN06}) and regular trees (see \cite{CY99}, \cite{CMS00} and \cite{CJK15}).  These formulas will be
rederived below from the general methodology we develop in this article (see Example \ref{ex. dirac seed tree}). For instance, let $G$ be a $(q+1)$-regular tree, which has special significance since it is the universal covering of every $(q+1)$-regular graph, then we have
\begin{equation*}
H_{G}(x,y;t)=q^{-r/2}e^{-(q+1)t}I_r(2\sqrt{q}t) - (q-1)e^{-(q+1)t}\sum_{j=1}^{\infty}q^{-(r+2j)/2}I_{2j+r}(2\sqrt{q}t).
\end{equation*}
where $I_n(t)$ is the classical $I$-Bessel function. This is the formula in \cite{CMS00, CJK15}.

\subsubsection{Applications of uniqueness property}

As already mentioned, comparing different heat kernel expressions often leads to significant consequences.
and \cite{JKS24} that use the heat kernel on the lattice graph and its quotient to obtain explicit evaluations of trigonometric sums.  See also  \cite{HSSS23} for series identities of $I$-Bessel functions, from which the authors deduce transformation formulas for the Dedekind eta-function.

By choosing a different metric on $G$, and a few examples are presented in section \ref{sec: metric} below,
we get different ``seeds'' $H_d$ in the parametrix construction. Given that in our setting the heat kernel is unique,
when equating the expressions for the heat kernels constructed using  $H_d$ with different distances $d$ yields many new
identities.  Going further, taking the Laplace transform of such identities may
be particularly useful since the Laplace transform of a time convolution of two functions equals the
product of Laplace transforms.

When the Laplacian extends to a bounded, self-adjoint operator on $L^2(\theta)$, the heat kernel possesses a spectral expansion in terms of the eigenvalues and the eigenfunctions of the Laplacian.
As such, when using the parametrix construction of the heat kernel and starting with the seed $H_d$, which is of a geometric nature,
one can also deduce identities relating the sum over the eigenvalues to the length spectrum of the graph.  This, in turn,
may serve as a starting point for deducing trace-type formulas, see for example \cite{CY97},  \cite{TW03}, \cite{HNT06} or \cite{Mn07}.
A similar argument in the setting of compact hyperbolic Riemann surfaces yields a succinct proof of the Selberg trace formula;  see Remark 3.3 of \cite{GJ18}.

\subsubsection{Applications to Gaussian heat kernel estimates}

Gaussian-type heat kernel estimates on graphs are very important tool for studying existence/nonexistence and behavior of global solutions of heat-type equations (such as linear or semilinear heat equations) on graphs, see e.g. \cite{Wu18}, \cite{Wu21} or \cite{Li24}. For this reason, it is of interest to derive such estimates, under certain assumptions on the geometry of a graph, see e.g. recent papers \cite{HLLY19}, \cite{AH21} or \cite{Wa23}.

Under reasonably mild assumptions (G1), (G2) and (G3') on a graph $G$, it is proved that for any metric $d$ on $G$, uniformly bounded from below, the function $H_d$ defined by \eqref{eq. defn param exp} can be taken as a parametrix. Therefore, the heat kernel on $G$ is expressed as a the following function on $VG\times VG \times [0,\infty)$ depending on the vertex and edge weights ($\theta$, $w$) and the metric $d$:
$$
H_G(x,y;t)=\frac{1}{\sqrt{\theta(x)\theta(y)}} \exp(-(\theta(x)\theta(y)d^2(x,y))/t)+ \mathcal{F}(x,y;t),
$$
where
$$
\mathcal{F}(x,y;t)= \sum_{\ell=1}^{\infty}(-1)^{\ell}\left(H_d\ast(L_{G,x}H_d)^{\ast\ell}\right)(x,y;t).
$$
One may then produce Gaussian-type bounds on convolution sums by reasoning analogously as in \cite{Ch84}, pp. 153--154. This application, which is well-suited for obtaining on-diagonal Gaussian-type bounds for any metric uniformly bounded from below, and which complements findings of \cite{HLLY19}, \cite{AH21} and \cite{Wa23} will be developed in a subsequent article.

\subsubsection{Construction of the resolvent kernel}

The resolvent kernel, or the Green's function, associated to the graph Laplacian can be expressed as the Laplace transform of the heat kernel (possibly truncated by the contribution from the eigenvalue zero), see e.g. \cite{CY00}.  Applying the formula 3.478.4 of \cite{GR07} with $\nu=p=1$, the expression for the heat kernel with the  ``seed'' $H_d$ can be used in order to deduce an explicit expression for the resolvent kernel in terms of series of products of $K$-Bessel functions with index $1$. We will exploit this construction of the resolvent kernel in a subsequent article.

In case one is interested in computing the resolvent kernel of a subgraph $G'$ of an infinite graph $G$, one may start with the heat kernel $H_G$ on $G$ as a parametrix and use it to derive the heat kernel on $G'$ as described in Theorem \ref{thm: formula for heat kernel} below. Then, by taking the Laplace transform of this expression (possibly truncated by the contribution from the eigenvalue zero) one will be able to express the resolvent kernel on $G'$ in terms of the resolvent kernel on $G$, see e.g. \cite{LNY23} for the case of a subgraph of a complete graph.

\subsubsection{Varying edge and vertex weights}

For a given graph $G$, with the edge weights function $w$ let $H_{G,w}$ be the heat kernel associated to the edge weights function $w$.  Let $w'$
denote another edge weights function on $G$, and $H_{G,w'}$ be the associated heat kernel.  Then
$H_{G,w}$ can be used as a parametrix for the heat kernel on $(G,w')$.  In fact, the convolution
series in Theorem \ref{thm: formula for heat kernel} gives a precise formula for the difference
$$
H_{G,w'}(x,y;t) - H_{G,w}(x,y;t).
$$
From this expression, it is possible to study the variation of the heat kernel in $w$, from
which one could study spectral invariants derived from the heat kernel such as regularized determinants
or asymptotic behavior as, say, one edge weight approaches zero.

For a given graph $G$, with the vertex weights functions $\theta$, $\theta'$ let $H_{G,\theta}$, $H_{G,\theta'}$ be the corresponding heat kernels. Then,
$\sqrt{\frac{\theta(x)\theta(y)}{\theta'(x)\theta'(y)}}H_{G,\theta}(x,y;t)$ can be used as a parametrix for the heat kernel on $(G,\theta')$.

\subsection{Organization of the paper}

This paper is organized as follows. We start in section \ref{ref our setup} by describing the assumptions on the graph $G$, followed by section \ref{sec: conv on a graph} where we define the convolution on the space
$L^p(\theta)$ and prove some of its properties, which ensure that a certain convolution series is convergent. In section \ref{sec:The-parametrix-construction} we define the notion of a parametrix and prove our main theorem, Theorem
\ref{thm: formula for heat kernel}, which computes the heat kernel by starting with a parametrix.
We then proceed by constructing different parametrix and their associated heat kernels. In section \ref{sec. Dirac delta param}
it is proved that the Dirac delta on $L^1(\theta)$ can be taken as the parametrix, under the assumption that the graph $G$ is locally finite. In section \ref{sec. further ex} we develop further examples, under assumption that the combinatorial
vertex degree is uniformly bounded on $G$.  In particular,
we show that the function
$H_d$ defined above can be taken as a parametrix for an arbitrary choice of metric $d$ on $G$ provided $d$  is uniformly bounded
from below.
We conclude the paper with several remarks discussing other natural conditions posed on a graph.

\section{Our setup}\label{ref our setup}

Let $G$ be a countably infinite, vertex-weighted, edge-weighted, and connected graph with the vertex set
$VG$.  The vertex weights are defined by the function $\theta: VG \to (0,\infty)$ with full support,
meaning $\theta(x) > 0$ for all $x \in VG$.
The edge weights are defined by a nonnegative function $w:VG\times VG\to [0,\infty)$, and
we shall use the notation $w_{xy} = w(x,y)$.
We assume that $w$ is symmetric, meaning that $w_{xy} = w_{yx}$ for all $x,y\in VG$, hence the graph is undirected.
If $w_{xy}=0$ we say there is no edge between the vertices $x$ and $y$.
When $w_{xy}>0$, the vertices $x$ and $y$ are said to be adjacent or neighbors, and we write $x\sim y$.
Furthermore, we assume that $w_{xx}=0$ for all $x$, meaning that $G$ has no loops. We refer to \cite{Fo13} p. 117 as well as \cite{Fo14} for an interesting probabilistic interpretation of the edge weights and the vertex weights.

Let $p\in[1,\infty]$ be a real number or $+\infty$, and let $L^p(\theta)$ denote the classical $L^p$ space of
functions on $VG$ with respect to the pointwise vertex measure $\theta$.  Specifically, the function $f:VG \to \mathbb{C}$ belongs
to $L^p(\theta)$ for $p\in (1,\infty)$ if and only if
$$
\sum_{x\in VG} |f(x)|^p\theta(x)<\infty.
$$
Since $\theta$ has a full support, any $f\in L^{\infty}(\theta)$ is such that $f$ is uniformly bounded on $VG$.
We denote by $\|f\|_{p,\theta}$ the $L^p$-norm of $f$.
When $p=2$, the space $L^2(\theta)$ is a Hilbert space with the inner product
$$
\langle f,g \rangle_\theta:= \sum_{x\in VG} f(x)\overline{g(x)}\theta(x).
$$

Following \cite{KL12}, we define the Laplace operator $\Delta_G$ acting on functions $f:VG\rightarrow\mathbb{C}$ by
\begin{equation} \label{eq. lapl operator}
\Delta_{G}f(x)=\frac{1}{\theta(x)}\sum_{y\in VG}(f(x)-f(y))w_{xy}.
\end{equation}
The natural domain of definition is $\mathcal{D}(\Delta_G):= \{f\in L^2(\theta):\, \Delta_G f \in  L^2(\theta) \}$.

For any $x \in VG$, define
$$
\mu(x):= \sum_{y:\, x\sim y} w_{xy}.
$$
For us, we will assume throughout this paper that $\theta$ need not equal $\mu$.
In general, the Laplacian $\Delta_G$ is not bounded, nor possesses a unique self-adjoint extension.
For this reason, different authors imposed the following assumptions on the (infinite) graph $G$. Throughout this paper, we assume that the graph $G$ is locally finite, meaning that $\mu(x)<\infty$ for all $x\in VG$. Moreover, we assume that the operator $\Delta_G$ is bounded on $L^2(\theta)$, which (see  page 66 of \cite{Da93}, \cite{HKLW12} or \cite[Theorem 1.27]{KLW21}) is equivalent to posing the following assumption
\begin{itemize}
\item [(G1)] \textbf{Boundedness of the Laplacian}. There exists a positive constant $M$ such that
\begin{equation}\label{eq:A_definition}
A(G,w,\theta):=\sup\limits_{x\in VG} \frac{\mu(x)}{\theta(x)}\leq M <\infty.
\end{equation}

\end{itemize}

For the weights, we assume
\begin{itemize}
\item[(G2)] {\bf Uniform lower bound for the vertex weights.}
There exists $\eta>0$ such that
$$
\inf_{x\in VG} \theta(x) >\eta,
$$
\end{itemize}
a condition which arises naturally in \cite{Fo14} and \cite{Wu18}.

\vskip .06in
\textit{Throughout this paper, the weighted graph $G$ satisfies conditions} (G1) \textit{and} (G2).

\vskip .06in
The \emph{heat kernel} on the
graph $G$ associated to the weighted graph Laplacian $\Delta_{G,x}$, when
acting on functions in the variable $x$, is the unique bounded
solution $H_{G}(x,y;t)$
to the differential equation
\[
\left(\Delta_{G,x}+\frac{\partial}{\partial t}\right)H_{G}(x,y;t)=0
\]
with the property that
\begin{equation}\label{eq. initial cond}
\lim_{t\to0}H_{G}(x,y;t)= \frac{1}{\theta(x)}\delta_{x=y}= \left\{
                            \begin{array}{ll}
                              \frac{1}{\theta(x)}, & \text{  if  } x=y \\
                              0, & \text{  if  } x\neq y
                            \end{array}
                          \right. ,
\end{equation}
where $\delta$ is the Kronecker delta function.

For the proof of the existence and uniqueness of the heat kernel on
$G$, we refer to  \cite{Do84}, \cite{DM06} for the case when $\theta(x)\equiv 1$ and to
\cite{Hu12}, in a general setting. Let us note that the uniqueness essentially follows
from the fact that (G1) and (G2) imply that $\Delta_G$ is stochastically complete; see
\cite[Theorem 4.3]{Wo21} as well as the extensive bibliography
therein.
Specifically, according to \cite[Theorem 3.3]{Wo21}, stochastic completeness is equivalent to unique dependence of the
bounded solution to the heat equation from its initial condition which are given by a bounded function on $G$.

\section{Convolution on $L^p(\theta)$    \label{sec: conv on a graph}}

For any $p\in[1,\infty]$ we denote by $q$ its conjugate, meaning the number $q\in[1,\infty]$  such that $\frac{1}{p} + \frac{1}{q}=1$,
or $p=1$ if $q=\infty$.
Let us start by defining the convolution of two functions that belong to conjugate $L^p(\theta)$ spaces.

\begin{definition}\label{def. conv}
With the notation as above, let $F_{1},F_{2}:VG\times VG\times\mathbb{R}_{>0}\to\mathbb{R}$ be two functions
such that for any $t>0$, those functions, when viewed as functions on $VG\times VG$, belong to $ L^p(\theta)$
in the second variable and $L^q(\theta)$ in the first variable, respectively. Assume further that for all
$b>0$ and for all $x,y \in VG$, the function $\langle F_1(x,\cdot;t-r), F_2(\cdot, y;r)\rangle_{\theta}$,
which is well defined, due to the H\"older inequality, is integrable on $[0,b]$. The \emph{convolution}
of functions $F_{1}$ and $F_{2}$ is
defined to be
\begin{align*}
(F_{1}\ast F_{2})(x,y;t)&:=\int\limits _{0}^{t} \langle F_1(x,\cdot;t-r), F_2(\cdot, y;r)\rangle_{\theta}dr \notag \\ &=\int\limits _{0}^{t}\sum_{z\in VG}F_{1}(x,z,t-r)F_{2}(z,y;r) \theta(z)dr.\label{def: convolution}
\end{align*}
\end{definition}

The above convolution is not commutative in general but it is associative, under suitable assumptions on functions convolved.
We have the following lemma.


\begin{lem}
\label{lem:convolution bounds} Let $F_{1},F_{2}:VG\times VG\times\mathbb{R}_{>0}\to\mathbb{R}$
be as in Definition \ref{def. conv}.  For some $t_{0} > 0$, assume there exist constants $C_1, C_2$ and integers $k,\ell\geq 0$
such that for all $0<t<t_{0}$ and all $x,y \in VG$, we have
$$
\left\|F_{1}(x,\cdot;t)\right\|_{p,\theta}\leq C_{1}t^{k}
  \,\,\,\,\,
  \text{\rm and}
  \,\,\,\,\,
  \left\|F_{2}(\cdot, y;t)\right\|_{q,\theta}\leq C_{2}t^{\ell}.
$$
Then, for all $x,y\in VG$
we have
\[
  |(F_{1}\ast F_{2})(x,y;t)|\leq C_{1}C_{2}\frac{k!\ell!}{(k+\ell+1)!}
  t^{k+\ell+1}
  \,\,\,\,\,
  \text{\rm for $0<t<t_{0}$.}
\]
\end{lem}

\begin{proof}
From the H\" older inequality we get
\begin{align*}
|(F_{1}\ast F_{2})(x,y;t)|&\leq \int\limits _{0}^{t} \left\| F_1(x,\cdot;t-r)\right\|_{p,\theta} \left\| F_2(\cdot, y;r)\right\|_{q,\theta}dr\\
&\leq C_1C_2\int_0^t r^k (t-r)^\ell\, dr=C_1C_2\frac{k!\ell!t^{k+\ell+1}}{(k+\ell+1)!},
\end{align*}
as claimed.
\end{proof}

Let $f=f(x,y;t):VG\times VG\times\mathbb{R}_{>0}\to\mathbb{R}$ be a function with the following property.  For any  $T>0$, for all $t\in (0,T]$,
and all arbitrary, but fixed, $x,y\in VG$ the functions $f(\cdot, y;t): VG \to \mathbb{R} $ and $f(x,\cdot;t): VG \to \mathbb{R} $
belong to $L^p(\theta)\cap L^q(\theta)$ and the function $f(x,y;\cdot) :(0,T] \to \mathbb{R}$ is integrable.
For any positive integer $\ell$ and any such function $f$, we can inductively define the $\ell$-fold
convolution $(f)^{\ast\ell}(x,y;t)$ for $t\in(0,T]$ by setting $(f)^{\ast1}(x,y;t)=f(x,y;t)$
and, for $\ell\geq2$ we put
$$
(f)^{\ast\ell}(x,y;t):=\left(f\ast(f)^{\ast(\ell-1)}\right)(x,y;t),
$$
under additional assumption that $(f)^{\ast(\ell-1)}(x,\cdot;t)\in L^q(\theta)$ for all $t\in(0,T]$ and all $\ell \geq 2$.

With this notation we have the following lemma.

\begin{lem}
\label{lem:convergence of the series} Let $f=f(x,y;t):VG\times VG\times\mathbb{R}_{>0}\to\mathbb{R}$.
Assume that for all $x,y\in VG$ and all $t_{0} \in \mathbb{R}_{>0}$, the function $f(x,y;\cdot)$ is integrable on
the interval $(0,t_{0}]$ and $f(x,\cdot;t) \in L^1(\theta)$ for all $t\in (0,t_0]$.
Assume further that for all $t_{0} \in \mathbb{R}_{>0}$ there exists a constant $C$, depending only upon $t_0$, and integer $k\geq 0$ such that
$$
|f(x,y;t)|\leq Ct^{k}
\,\,\,\,\,
\text{\rm for all $x,y\in VG$ and $0<t<t_{0}$.}
$$
Then the series
\begin{equation}\label{eq: series over ell}
\sum_{\ell=1}^{\infty}(-1)^{\ell}(f)^{\ast\ell}(x,y;t)
\end{equation}
converges absolutely and uniformly on every compact subset of $VG\times VG\times\mathbb{R}_{>0}$.
In addition, we have that
\begin{equation}
\left(f\ast\left(\sum_{\ell=1}^{\infty}(-1)^{\ell}(f)^{\ast\ell}\right)\right)(x,y;t)=
\sum_{\ell=1}^{\infty}(-1)^{\ell}(f)^{\ast(\ell+1)}(x,y;t).\label{eq: convolution with inf sum}
\end{equation}
and
\begin{equation}\label{eq:series_bound}
\sum_{\ell=1}^{\infty}\left|(f)^{\ast\ell}(x,y;t)\right|
=O(t^{k})
\,\,\,\,\,
\text{\rm as $t\to 0$,}
\end{equation}
where the implied constant is independent of $x,y\in VG$.
\end{lem}

\begin{proof} Let $A$ be an arbitrary compact subset of $VG\times VG\times\mathbb{R}_{>0}$.
Let $t_0>0$ be such that $A\subseteq VG\times VG\times (0,t_0]$.
We apply Lemma \ref{lem:convolution bounds}, with $p=\infty$, $q=1$ to get that
\begin{equation*}
  \label{eq:4}
\left|(f\ast f)(x,y;t)\right| \leq C \|f\|_{1,\theta}\frac{t^{k+1}}{(k+1)!}, \,\,\,\,\,
\text{\rm for all $x,y\in VG$ and $0<t<t_{0}$.}
\end{equation*}
Similarly, by induction for $\ell\geq 1$ we have the bound that
\begin{equation*}
  \label{eq:mult_convolution_bound}
\left|(f^{\ast(\ell)})(x,y;t)\right|
\leq
C\|f\|_{1,\theta}^{\ell-1}\frac{t^{k+\ell-1}}{(k+\ell-1)!}, \,\,\,\,\,
\text{\rm for all $x,y\in VG$ and $0<t<t_{0}$.}
\end{equation*}
The assertion regarding the convergence of (\ref{eq: series over ell}) now follows
from the Weierstrass criterion and the fact that $A\subseteq VG\times VG\times (0,t_0]$.

Fix $t>0$. The series (\ref{eq: series over ell}) converges absolutely.  When viewed as
a function of $y$, for any arbitrary but fixed $x$ and for any $0<t<t_0$, the series belongs to $L^{\infty}(\theta)$. Therefore,
\begin{equation}\label{eq. conv intergcange int sum}
\left(f\ast\left(\sum_{\ell=1}^{\infty}(-1)^{\ell}(f)^{\ast\ell}\right)\right)(x,y;t)= \int\limits_0^t \sum_{z\in VG} f(x,z;t-r) \left( \sum_{\ell=1}^{\infty} (-1)^\ell (f)^{\ast \ell}(z,y;r)\right) \theta(z)  dr.
\end{equation}
From the bound \eqref{eq:mult_convolution_bound}, combined with the H\" older inequality with $p=1$, $q=\infty$ we have, for an arbitrary $t\in(0,t_0)$ and $0<r<t$ that
$$
\sum_{\ell=1}^{\infty}\sum_{z\in VG}\left| f(x,z;t-r) (f)^{\ast \ell}(z,y;r)\right|\theta(z)\leq C\|f\|_{1,\theta}t^k \exp(t \|f\|_{1,\theta}).
$$
Hence we may interchange the sum over $\ell$ with the sum over $z \in VG$ in \eqref{eq. conv intergcange int sum}. Reasoning analogously, one easily shows that we may interchange the infinite sum over $\ell$ with the integral from $0$ to $t$, to deduce that
$$
\left(f\ast\left(\sum_{\ell=1}^{\infty}(-1)^{\ell}(f)^{\ast\ell}\right)\right)(x,y;t)= \sum_{\ell=1}^{\infty} (-1)^\ell \int\limits_0^t \sum_{z\in VG} f(x,z;t-r)(f)^{\ast \ell}(z,y;r)\theta(z)  dr,
$$
which proves \eqref{eq: convolution with inf sum}.

Finally, the bound \eqref{eq:mult_convolution_bound} and the fact that the series  \eqref{eq: series over ell} converges absolutely on $VG\times VG\times (0,t_0]$ yield that
$$
\sum_{\ell=1}^{\infty}\left|(f)^{\ast\ell}(x,y;t)\right|\leq C\|f\|_{1,\theta}t^k \exp(t \|f\|_{1,\theta}),
$$
which proves \eqref{eq:series_bound}.
\end{proof}

\section{The parametrix construction of the heat         
  kernel on $G$\label{sec:The-parametrix-construction}}  

The heat operator $L_G$ on the graph $G$ is defined by
\begin{equation*}
  \label{def:heat_operator}
  L_G=\Delta_{G}+\frac{\partial}{\partial t}.
\end{equation*}
With this notation, the heat kernel $H_{G}$ on $G$ associated to the Laplacian $\Delta_G$ is the unique bounded solution
$H_G:VG\times VG\times [0,\infty)$ to the differential equation
\[
L_{G,x}H_{G}(x,y;t)=0
\]
satisfying the initial condition \eqref{eq. initial cond}.
The subscript $x$ on $L_{x}$ indicates the sum \eqref{eq. lapl operator} which defines the Laplacian is
over neighbors of $x$, the first space variable.

\begin{definition}\label{def. parametrix}
Let $k\geq 0$ be an integer. A \emph{parametrix} $H$ order $k$  for the heat operator $L_{G}$ on $G$ is any continuous function $H=H(x,y;t):VG\times VG\times [0,\infty)$
which is smooth in time variable $t$, integrable in each space variable, and satisfies the following properties.
\begin{enumerate}
\item For all $x,y\in VG$,
\begin{equation}
H(x,y;0) = \lim\limits_{t \rightarrow 0}
H(x,y;t)=\frac{1}{\theta(x)}\delta_{x=y}. \label{eq:dirac property of heat kernel}
\end{equation}
\item The function $L_{G,x}H(x,y;t)$ extends to a continuous function on  $VG\times VG\times [0,\infty)$.

 \item  For all  $x\in VG$, $t>0$ we have that $\Delta_{G,x}H(x,\cdot;t)$ and $\frac{\partial}{\partial t}  H(x,\cdot;t)$ are in $L^1(\theta)$.
\item For any $t_0>0$ there exists a constant $C=C(t_0)$, depending only on
$t_0$, such that
$$
|L_{G,x}H(x,y;t)| \leq C(t_0)t^k
\,\,\,\,\,
\text{\rm for $t \in (0, t_0]$}
\,\,\,
\text{\rm and all}
\,\,\,
\text{\rm $x,y\in VG$.}
$$
\end{enumerate}
\end{definition}
Note that the third assumption on the parametrix $H$ implies that $L_{G,x}H(x,\cdot;t)\in L^1(\theta)$.

\begin{lem}
\label{lem: L of convol} Let $H$ be a parametrix for the heat operator on $G$ of any order.
Let $f=f(x,y;t):VG\times VG\times\mathbb{R}_{> 0}\to\mathbb{R}$
be a continuous function in $t$ for all $x,y\in VG$. Assume further that for all $t>0$ the function
$f(x,y;t)$ when viewed as a function on $VG\times VG$ is uniformly bounded.  Then
$$
L_{G,x}( H\ast f)(x,y;t)=f(x,y;t)+(L_{G,x}H \ast f)(x,y;t)
$$
for all $x,y\in VG$ and $t\in\mathbb{R}_{>0}$.
\end{lem}

\begin{proof} As stated $f(\cdot, y;t)\in L^{\infty}(\theta)$ and $H(x,\cdot;t), \, \Delta_{G,x}H(x,\cdot;t)\in L^1(\theta)$
for all $(x,y;t) \in VG\times VG\times (0,\infty)$.  Therefore, the convolutions $H\ast f$ and $(\Delta_{G,x}H)\ast f$ are well defined,
and we have that
\begin{align}\notag
L_{G,x}(H\ast f)(x,y;t)&=
    \frac{\partial}{\partial t}(H\ast f)(x,y;t)
    + \Delta_{G,x}(H\ast f)(x,y;t)\\ &=
  \frac{\partial}{\partial t}(H\ast f)(x,y;t)+  \left( (\Delta_{G,x}H)\ast f\right)(x,y;t),\label{eq:heat_of_convolution}
\end{align}
where the second equation follows from the fact that the Laplacian acts on the first variable only.

The function $\sum_{z\in VG}H(x,z;t-r)f(z,y;r) \theta(z)$ is continuous in the time variable, so we can apply the
Leibniz integration formula.  Upon doing so, we obtain that
the first term on the right hand side of \eqref{eq:heat_of_convolution} is equal to
\begin{align} \label{eq.deriv in t}
\frac{\partial}{\partial t}&\int\limits _{0}^{t}\sum_{z\in VG}H(x,z;t-r)f(z,y;r) \theta(z)dr\\&
=\sum_{z\in VG}H(x,z;0)f(z,y;t)\theta(z)+\int\limits _{0}^{t} \frac{\partial}{\partial t} \sum_{z\in VG}
H(x,z;t-r)f(z,y;r) \theta(z)dr.\notag
\end{align}
The assumptions that $f$ is uniformly bounded and that $\frac{\partial}{\partial t}  H(x,\cdot;t)\in L^1(\theta)$ combine
to yield that
$$
\frac{\partial}{\partial t} \sum_{z\in VG}
H(x,z;t-r)f(z,y;r) \theta(z)= \sum_{z\in VG}\frac{\partial}{\partial t} H(x,z;t-r)f(z,y;r) \theta(z).
$$
Given that $H(x,z;0)=0$ unless $x=z$, we get from \eqref{eq.deriv in t} that
$$
\frac{\partial}{\partial t}\int\limits _{0}^{t}\sum_{z\in VG}H(x,z;t-r)f(z,y;r) \theta(z)dr= f(x,y;t) + \left(\frac{\partial}{\partial t}H\ast f\right)(x,y;t).
$$
Therefore,
\begin{align*}
L_{G,x}(H\ast f)(x,y;t) &=
  \frac{\partial}{\partial t}(H\ast f)(x,y;t)+   (\Delta_{G,x}H\ast f)(x,y;t)
\\&=f(x,y;t)+\left(\frac{\partial}{\partial t}H\ast f\right)(x,y;t)+(\Delta_{G,x}H\ast f)(x,y;t)
\\&=f(x,y;t)+(L_{G,x}H\ast f)(x,y;t),
\end{align*}
as claimed.
\end{proof}

With all this, we now can state the main theorem in this section.

\begin{thm}
\label{thm: formula for heat kernel} Let $H$ be a parametrix of
order $k\geq0$ for the heat operator on $G$.  For $x,y\in VG$ and $t\in \mathbb{R}_{\geq 0}$, let
\begin{equation}\label{eq:Neuman_series}
F(x,y;t):=\sum_{\ell=1}^{\infty}(-1)^{\ell}
(L_{G,x}H)^{\ast\ell}(x,y;t).
\end{equation}
Then the Neumann series \eqref{eq:Neuman_series}
converges absolutely and uniformly on every compact subset of $VG\times VG\times\mathbb{R}_{\geq 0}$.
Furthermore, the heat kernel $H_{G}$ on $G$
associated to graph Laplacian $\Delta_{G,x}$ is given by
\begin{equation}\label{eq:heat_kernel_parametrix}
H_{G}(x,y;t)=H(x,y;t)+(H\ast F)(x,y;t)
\end{equation}
and
$$
(H\ast F)(x,y;t)
= O(t^{k+1})
\,\,\,\,\,
\text{\rm as $t \rightarrow 0$.}
$$
\end{thm}

\begin{proof}
Set
$$
\tilde{H}(x,y;t) := H(x,y;t)+(H\ast F)(x,y;t).
$$
We want to show that
\begin{equation}\label{eq:heat_kernel_criteria}
L_{G,x}\tilde{H}(x,y;t)=0
\,\,\,\,\,
\text{\rm and}
\,\,\,\,\,
\lim_{t\to 0} \tilde{H}(x,y,t)=\frac{1}{\theta(x)}\delta_{x=y}.
\end{equation}
By Lemma \ref{lem:convergence of the series}, the series $F(x,y;t)$ defined in \eqref{eq:Neuman_series}
converges uniformly and absolutely and has order $O(t^{k})$ as $t \rightarrow 0$.  Since $H$ is in $L^1(\theta)$,
 Lemma \ref{lem:convolution bounds} then yields the asymptotic bound that
$$
(H\ast F)(x,y;t)= O(t^{k+1})
\,\,\,\,\,
\text{\rm as $t \rightarrow 0$.}
$$
Therefore,
$$
\lim_{t\to 0}\widetilde{H}(x,y;t)=\lim_{t\to 0}H(x,y;t)=\frac{1}{\theta(x)}\delta_{x=y}.
$$
It remains to prove the vanishing of $L_{G,x}\tilde{H}$ in \eqref{eq:heat_kernel_criteria}.
For this, we can apply Lemma \ref{lem: L of convol} to get that

\begin{align*}
L_{G,x}\tilde{H}(x,y;t)&=
L_{G,x}H(x,y;t)+L_{G,x}(H\ast F)(x,y;t)
\\&=
L_{G,x}H(x,y;t)+\sum_{\ell=1}^{\infty}(-1)^{\ell}(L_{G,x}
{H})^{\ast\ell}(x,y;t)\\ &\,\,\,\,\,\, +(L_{G,x}{H})*\left(\sum_{\ell=1}^{\infty}(-1)^{\ell}(L_{G,x}{H})^{\ast(\ell)}\right)(x,y;t)\\&=
L_{G,x}H(x,y;t)+\sum_{\ell=1}^{\infty}(-1)^{\ell}(L_{G,x}
{H})^{\ast\ell}(x,y;t)+\sum_{\ell=1}^{\infty}(-1)^{\ell}(L_{G,x}{H})^{\ast(\ell+1)}(x,y;t)\\&
=0,
\end{align*}
To be precise, in the above calculations we used
that absolute convergence of the series defining $F(x,y;t)$
in order to change the order of summation.  This completes the proof.
\end{proof}

\section{Dirac delta as a parametrix}\label{sec. Dirac delta param}

In \eqref{eq:heat_kernel_parametrix}, the function $H(x,y;t)$ is a
parametrix, so it is required to satisfy the reasonably weak conditions given in its definition.  In particular, one does not use any information about the edge structure
associated to the graph. However, the edge data is essential in the definition of the Laplacian, which is used in the construction of the series \eqref{eq:Neuman_series} through the heat operator. In this section we highlight the role of the edge data in the construction.

We assume that an infinite, locally finite graph $G$ satisfies assumptions (G1), (G2) and the following additional assumption.
\begin{itemize}
\item [(G3)] {\bf Finiteness of the combinatorial vertex degree.}
For all $x\in VG$ the number of $y\in VG$ such that $w_{xy}>0$ is finite.
\end{itemize}

We have the following proposition.

\begin{prop}\label{prop. dirac delta HK}
Let $G$ be a connected, locally finite, undirected infinite graph satisfying assumptions (G1), (G2) and (G3) above. Let $H(x,y;t)$ be a function on
$VG\times VG\times [0,\infty)$ defined for all $x,y \in VG$ and all $t\in[0,\infty)$ by
\begin{equation}\label{eq:delta_parametrix}
H(x,y;t) :=\frac{1}{\theta(x)}\delta_{x=y}.
\end{equation}
Let $\delta_x(y)$
be defined as in \eqref{eq:delta_definition}.
Then, $H(x,y;t)$ defined by \eqref{eq:delta_parametrix} is the parametrix of order $k=0$, and the heat kernel $H_G$ constructed when using
\eqref{eq:delta_parametrix} as a parametrix is given by \eqref{eq. HK from dirac}.
\end{prop}

\begin{proof}
The set $VG$ is discrete, so then $H(x,y;t)$ is continuous on $VG\times VG\times [0,\infty)$ and smooth in $t$ for fixed $x$ and $y$.
It is evident that $H$ belongs to $L^1(\theta)$ in both space variables and, furthermore, satisfies the initial
condition \eqref{eq. initial cond}.  Trivially,
\begin{equation}\label{eq. L of Dirac}
L_{G,x}H(x,y;t)= \Delta_{G,x}H(x,y;t) =\frac{1}{\theta^2(x)}\left\{
                                                             \begin{array}{ll}
                                                               \mu(x), & x=y; \\
                                                               -w_{xy}, & x\sim y \\
                                                               0, & \text{otherwise.}
                                                             \end{array}
                                                           \right.
\end{equation}
The function \eqref{eq. L of Dirac} is continuous on $VG\times VG\times [0,\infty)$,
hence the second condition for the parametrix in Definition \ref{def. parametrix} is fulfilled.
Moreover, $\Delta_{G,x}H(x,\cdot;t)\in L^1(\theta)$ because the sum
$$
\sum_{z\in VG} \theta(z)\Delta_{G,x}H(x,z;t)
$$
is a finite sum, by the assumption (G3). Since $\frac{\partial}{\partial t} H(x,\cdot; t) \equiv 0$, we conclude
that the third condition for the parametrix is fulfilled. Finally, from \eqref{eq. L of Dirac} it is evident
that for all $x,y \in VG$  and all $t>0$ we have
$$
|L_{G,x}H(x,y;t)| \leq \frac{\mu(x)}{\theta^2(x)} \leq \frac{M}{\eta}.
$$
Therefore, $H(x,y;t)$ is a parametrix of order $k=0$.

The heat kernel, as constructed in Theorem \ref{thm: formula for heat kernel} when using the parametrix $H$, is given by
\begin{equation}\label{eq:heat_kernel dirac}
H_{G}(x,y;t)=H(x,y;t)+ \sum_{\ell=1}^{\infty}(-1)^{\ell}
\left(H\ast(L_{G,x}H)^{\ast\ell}\right)(x,y;t)
\end{equation}
We have the following evaluation of the first few terms in the convolution series \eqref{eq:heat_kernel dirac}.  First,
\begin{align*}
(H\ast L_{G,x}H)(x,y;t) &= \int\limits_0^t \sum_{z\in VG} \theta(z) H(x,z;t-\tau)L_{G,z}H(z,y;\tau) d\tau = t \theta(x)\frac{ \delta_x(y)}{\theta(x)}
\\&  =t\delta_x(y).
\end{align*}
Using that $
\theta(z)L_{G,z}H (z,y;\tau) = \delta_z(y)
$ we get
\begin{align*}
(H\ast L_{G,x}H)\ast L_{G,x}H (x,y;t) &=\int\limits_0^t \sum_{z\in VG} \theta(z) (t-\tau)\delta_x(z)L_{G,z}H (z,y;\tau)d\tau\\
&=\frac{t^2}{2}\sum_{z\in VG} \delta_x(z) \delta_z(y).
\end{align*}
Next, when proceeding by induction, we deduce that
$$
\left(H\ast(L_{G,x}H)^{\ast\ell}\right)(x,y;t)=\frac{t^{\ell}}{\ell !}\sum_{z_1,\ldots,z_{\ell-1} \in VG}
 \delta_x(z_1)\delta_{z_1}(z_2)\cdot \ldots \cdot \delta_{z_{\ell-1}}(y).
$$
The sum on the right-hand side is finite, since $\delta_x(y)$ is supported on a finite set, due to (G3).

With all this, we conclude that the heat kernel on $G$ is given by \eqref{eq. HK from dirac}.
\end{proof}


\begin{example}
Consider the case when $G=\mathbb{Z}$, meaning the graph whose set of vertices is the
set of integers.  The two vertices $x,y \in \mathbb{Z}$ are connected if and only if $x-y\in\{-1,1\}$.  For
every $x \in G$, let the vertex weight be
$\theta(x) = 1$, and assume all edges weights are also equal to one.   Let us use \eqref{eq. HK from dirac} to
compute the heat kernel on $\mathbb{Z}$.

The product $ \delta_x(z_1)\delta_{z_1}(z_2)\cdot \ldots \cdot \delta_{z_{\ell-1}}(y)$ in \eqref{eq. HK from dirac}
is non-zero precisely when the sequence $x=z_0,z_1,\ldots, z_{\ell-1},
y=z_\ell \in \mathbb{Z}$ is such that $z_h-z_{h-1}\in\{-1,0,1\}$ for all $h=1,\ldots,\ell$.  Such a sequence
can be identified with an $\ell$-tuple $(a_1,\ldots, a_\ell)$
where $a_1,\ldots, a_\ell \in\{-1,0,1\}$.

For $x,y\in \mathbb{Z}$, let $j\geq 0$ be such that $x-y=j$; we will comment later when $x-y=-j$.
Then the $\ell$-tuple  $(a_1,\ldots, a_\ell)$ must have exactly $j$ places
all with the values $1$. Assume that $k\geq 0$ is the number of places $a_h$ in the $\ell$-tuple  $(a_1,\ldots, a_\ell)$ which are equal
to zero. Then at the remaining  $\ell-j-k$ places there must be the same number of entries with $1$ and with $-1$; in particular,
$k$ must be such that $\ell-j-k$ is an even number, say $i$, so we have that $\ell-j-k=2i$.

Therefore, every $\ell$-tuple $(a_1,\ldots, a_\ell)$ corresponding
to the sequence $x=z_0,z_1,\ldots, z_{\ell-1}, y=z_\ell \in \mathbb{Z}$ such that $z_h-z_{h-1}\in\{-1,0,1\}$ and such that $x-y=j$
is uniquely determined by $k$ places at which there are zeros, where $k$ is such that $\ell-k-j$ is even, and by $i=\frac{1}{2}(\ell -k-j)$
places at which there are the numbers $-1$.  The remaining places all have the value $1$.
The number of all such $\ell$-tuples of elements from $\{-1,0,1\}$ is exactly
$$
\binom{\ell}{k}\binom{\ell-k}{i}=\frac{\ell!}{k!i!(\ell-k-i)!},
$$
where $\ell=k+j+2i$. For each such $\ell$-tuple, we have
$$
 \delta_x(z_1)\delta_{z_1}(z_2)\cdot \ldots \cdot \delta_{z_{\ell-1}}(y)=2^k\cdot (-1)^{j}=(-1)^{\ell+k}2^k,
$$
where the last equality follows since $\ell-k-j$ is even, so $(-1)^j=(-1)^{\ell-k}=(-1)^{\ell+k}$. Therefore, for $x-y=j\geq 0$
$$
\sum_{z_1,\ldots,z_{\ell-1} \in VG} \delta_x(z_1)\delta_{z_1}(z_2)\cdot \ldots \cdot \delta_{z_{\ell-1}}(y)= \underset{\ell-k-j \text{  even}}{\sum_{i,k\geq 0\,\, k+j+2i=\ell}}(-1)^{\ell +k} \frac{\ell!}{k! i! (i+j)! }2^k.
$$
Finally, we get that
\begin{align*}
H_{\mathbb{Z}}(x,y;t) &= \sum_{\ell=0}^\infty t^\ell \underset{\ell-k-j \text{  even}}{\sum_{i,k\geq 0\,\, k+j+2i=\ell}} \frac{1}{k! i! (j+i)! }(-2)^k
\\&= \sum_{k=0}^{\infty}\frac{(-2t)^k}{k!} \cdot \sum_{i=0}^{\infty}\frac{ t^{j+2i}}{i! (j+i)!}= e^{-2t} I_{j}(2t),
\end{align*}
where $I_j(2t)$ is the $I$-Bessel function, see \cite{GR07}, formula 8.445 with $\nu=j$.

In the case $x-y=-j$, we reverse the roles of $1$ and $-1$ in the above combinatorial argument, we get
that $H_{\mathbb{Z}}(x,y;t) = e^{-2t} I_{j}(2t)$.
In summary, we have shown that from \eqref{eq. HK from dirac} one gets that
the heat kernel on $\mathbb{Z}$ is $H_{\mathbb{Z}}(x,y;t) = e^{-2t}I_{|x-y|}(2t)$, which was previously
established in Bednarchak's thesis or essentially already in \cite[p. 60]{Fe71}, see also \cite{CY97, Be03, KN06}.
\end{example}

\begin{example} \label{ex. dirac seed tree}
For $q \geq 2$, let
$G=T_{q+1}$ be a $q+1$-regular tree with vertex weights $\theta\equiv 1$.
Each vertex is connected to $q+1$ vertices with edges, and
we assume the edge weights are all equal to $1$.  Then $\delta_x(z)=(q+1)$ if $z=x$,  $\delta_x(z)=-1$ for the $q+1$ vertices $z$ adjacent
to $x$ and $\delta_x(z)=0$ otherwise.  We note that the case $q=1$ is considered in the
previous example.   Let us use \eqref{eq. HK from dirac} to
compute the heat kernel on $T_{q+1}$.

The product $ \delta_x(z_1)\delta_{z_1}(z_2)\cdot \ldots \cdot \delta_{z_{\ell-1}}(y)$ is non-zero if and only
if $x=z_0,z_1,\ldots, z_{\ell-1}, y=z_\ell $ are such that each pair of neighbouring entries are either
equal or adjacent. Let $x,y\in T_{q+1}$ be such that their distance is $r\geq 0$. If $\ell\leq r$, the product is obviously zero,
so we may assume that $\ell \geq r$.

For any integer $j$ with $0\leq j\leq \ell-r$, let us assume that exactly $j$ of the $\ell$ values $\delta_{z_i}(z_{i+1})$ are equal to $(q+1)$.
Those values can be chosen in $\displaystyle \binom{\ell}{j}$ ways.  For such selection of $j$ points $z_i=z_{i+1}$, the sequence
$x=z_0,z_1,\ldots, z_{\ell-1}, y=z_\ell $ becomes a walk of length $\ell-j$ from $x$ to $y$.  Let us denote the
number of such walks by $b_{\ell-j}(r)$. Note that he number of walks depends only on the distance $r$ between $x$ and $y$. Therefore,
when including the values taken by $\delta_{x}(z)$, we get that
\begin{align*}
a_\ell(r)&= \sum_{z_1,\ldots,z_{\ell-1} \in VG} \delta_x(z_1)\delta_{z_1}(z_2)\cdot \ldots \cdot \delta_{z_{\ell-1}}(x) = \sum_{j=0}^{\ell-r}\binom{\ell}{j} (q+1)^j (-1)^{\ell-j}b_{\ell-j}(r)\\
&= \sum_{j=r}^{\ell}\binom{\ell}{\ell-j} (q+1)^{\ell-j} (-1)^{j}b_{j}(r)= (q+1)^{\ell}\sum_{j=0}^{\ell}\binom{\ell}{j} (q+1)^{-j} (-1)^{j}b_{j}(r).
\end{align*}
In the last term, we have
adopted the convention that the number of walks of length $j<r$ between two points at distance $r$ to be equal to zero. Therefore,
$$
(-1)^{\ell}a_\ell(r)=\sum_{j=0}^{\ell}\binom{\ell}{j} (-(q+1))^{\ell-j} b_{j}(r),
$$
so then we have that
$$
H_{T_{q+1}}(x,y;t)=e^{-(q+1)t}\sum_{k=0}^{\infty}b_k(r)\frac{t^k}{k!}.
$$
Let us further evaluate this expression.

The ordinary generating function for the number of walks $b_k(r)$ of length $k$ on the $(q+1)$-regular tree, between two points at a distance $r$ is given by
\begin{equation}\label{eq:tree_o_generating}
f_{q+1}(t)=\frac{2q}{q-1+(q+1)\sqrt{1-4qt^2}}\left(\frac{1-\sqrt{1-4t^2}}{2qt}\right)^r,
\end{equation}
see \cite{McK83} and \cite{RZ09};
For  $r\geq 0$, the exponential generating function
$$
g_{q+1}(t):=e^{(q+1)t}H_{T_{q+1}}(x,y;t)
$$
of the sequence
$\{b_k(r)\}_{k=0}^{\infty}$ can be expressed in terms of the ordinary generating function in at least
two different ways.  In one approach, we start with the identity
$$
t^{-1}f_{q+1}(t^{-1})=(\mathcal{L} g_{q+1})(t),
$$
which is valid for $|t|>2\sqrt{q}$ and where $\mathcal{L}$ denotes the Laplace transform.
With elementary algebraic manipulations, one obtains the identity that
$$
(\mathcal{L} g_{q+1})(t)=\sum_{j=0}^{\infty}q^{-(r+2j)/2}\left(\frac{t-\sqrt{t^2-4q}}{2\sqrt{q}}\right)^{2j+r}\left(\frac{t-\sqrt{t^2-4q}}{2q}\right),
$$
which is valid for $|t|>2\sqrt{q}$. The identity
$$
\left(\frac{t-\sqrt{t^2-4q}}{2q}\right)= \frac{1}{\sqrt{t^2-4q}}\left(1-\left(\frac{t-\sqrt{t^2-4q}}{2\sqrt{q}}\right)^2\right)
$$
yields that
$$
(\mathcal{L} g_{q+1})(t)=\frac{1}{\sqrt{t^2-4q}}\sum_{j=0}^{\infty}q^{-(r+2j)/2}\left(\frac{t-\sqrt{t^2-4q}}{2\sqrt{q}}\right)^{2j+r}
\left(1 - q^{-1}\left(\frac{t-\sqrt{t^2-4q}}{2\sqrt{q}}\right)^{2}\right),
$$
which is valid for $|t|>2\sqrt{q}$. Therefore,
$$
(\mathcal{L} g_{q+1})(t)=\frac{q^{-r/2}}{\sqrt{t^2-4q}}\left(\frac{t-\sqrt{t^2-4q}}{2\sqrt{q}}\right)^r- (q-1)\sum_{j=1}^{\infty}\frac{q^{-(r+2j)/2}}{\sqrt{t^2-4q}}\left(\frac{t-\sqrt{t^2-4q}}{2\sqrt{q}}\right)^{2j+r}.
$$
From \cite{GR07}, section 17.13, formula 109 with $a=2\sqrt{q}$ and $\nu=r+2j$ for $j=0,1,\ldots$ and Re$(t)>2\sqrt{q}$, we have that
$$
\mathcal{L}(I_{r+2j}(2\sqrt{q}x))(t)=\frac{1}{\sqrt{t^2-4q}}\left(\frac{t-\sqrt{t^2-4q}}{2\sqrt{q}}\right)^{2j+r}.
$$
From this, we conclude that
\begin{equation}\label{eq:heat_kernel_tree}
e^{(q+1)t}H_{T_{q+1}}(x,y;t)=q^{-r/2}I_r(2\sqrt{q}t) - (q-1)\sum_{j=1}^{\infty}q^{-(r+2j)/2}I_{2j+r}(2\sqrt{q}t).
\end{equation}
The formula in \eqref{eq:heat_kernel_tree} is precisely
the expression for the heat kernel on a $T_{q+1}$ appearing in \cite{CMS00, CJK15}.

As stated, the exponential generating function $g_{q+1}(t)$ and the ordinary generating function $f_{q+1}(t)$ are
related by the Laplace transform, meaning that
$$
g_{q+1}(t)= \frac{1}{2\pi} \int\limits_{-\pi}^{\pi}f_{q+1}(te^{i\theta}) \exp(e^{i\theta}) d\theta,
$$
valid for $|t|< 1/(2\sqrt{q})$.  Using \eqref{eq:tree_o_generating}, one can employ elementary manipulations of
the integral and derive the integral expression for the heat kernel $H_{T_{q+1}}(x,y;t)$ established in \cite{CY99}.
We will omit the details of these computations.
\end{example}

\section{Using a distance metric to construct a parametrix}\label{sec. further ex}

In this section we will describe how to define further examples of a parametrix for the heat kernel using a distance function on
$G$.   The examples we develop come from different distance functions $d: VG \times VG \rightarrow \mathbb{R}_{\geq 0}$,
so we denote the parametrix by $H_{d}$.  In section \ref{sec: metric} we define
some metrics that can be used in the construction of $H_{d}$, and in section \ref{sec: ex of param} we give explicit examples
of an associated parametrix.
\vskip .06in

Throughout this section, we assume that an infinite graph $G$ satisfies assumptions (G1), (G2) and the following strengthening of the assumption (G3):
\begin{itemize}
\item [(G3')] {\bf Uniform boundedness of the combinatorial vertex degree.}
There exists a positive integer $N$ such that for all $x\in VG$ the number of $y\in VG$ such that $w_{xy}>0$ is bounded by $N$.
\end{itemize}

\noindent
Assumption (G3') is not unusual in the work that is related to properties of operators on infinite weighted graphs.
For example, it suffices to deduce natural upper and lower bounds for the heat kernel in terms of the combinatorial distance in \cite{MS00} and \cite{Sc02}. Moreover,  uniform boundedness of the combinatorial degree yields essential self-adjointness of the Laplacian, as well as of other operators on $G$, such as a  Schr\" odinger operator.
For more details, we refer to  \cite{CdVT-HT11} or \cite{Mi11}.

\subsection{Distance metric on a weighted graph} \label{sec: metric}

Let $d_G(x,y)$ denote an arbitrary distance metric on $G$ which is uniformly bounded from below.
More specifically, we assume there is a positive constant $\delta$ such that
\begin{equation}\label{eq. metric bound from below}
\text{\rm for all $x, y \in VG$, \,\,} x\neq y \,\, \text{  \rm we have that} \,\,\,\,\,  d_G(x,y)\geq \delta>0.
\end{equation}
Such a metric always exists on a connected graph, and below in this section, as well as
in section \ref{sec. edge weighted dist},
we will provide a few examples.

\subsubsection{Combinatorial graph distance}\label{sec:graph_metric}

The graph $G$ is connected, meaning for any two distinct points $x,y\in VG$ there exists a path connecting the $x$ to $y$.
To be clear, by a path we mean
a sequence of points $p(x,y)=\{x=x_0, x_1,\ldots, x_n= y\}$ such that $w_{x_jx_{j+1}}>0$ for all $j\geq 0$.
The path $p(x,y)=\{x=x_0, x_1,\ldots, x_n= y\}$ is of length $n\in \mathbb{N}$ . If $d_G(x,y)$ denotes the minimal length
of all paths that connect $x$ and $y$, it is then straightforward to conclude that $d_G$ is a distance metric on $G$.
Indeed, $d_G$ is well defined because $G$ is connected.  The function $d_{G}$ is symmetric, which follows from the fact that $G$ is undirected.
Also, the triangle inequality is immediate, while the equivalence $d_G(x,y)=0$ if and only if $x=y$ follows from the fact that $G$ has no loops.

This distance metric will be called the \emph{combinatorial graph distance}. It is obviously
satisfies \eqref{eq. metric bound from below} with $\delta=1$.

 The combinatorial graph distance is \emph{intrinsic} metric on $G$ (as defined in \cite{FLW14})  if and only if the condition (G3') is fulfilled, see \cite{HKMW13} and \cite{KLSW15}.

\subsubsection{Metric adapted to the Laplacian}

\vskip .06in
There are another choices for a metric on $G$ which in addition to satisfying \eqref{eq. metric bound from below} is closely related to stochastic properties of a graph.
For example, \emph{the normalized combinatorial graph distance} metric is defined by
$$
\rho_G:=(A(G,w,\theta))^{-1/2} d_G;
$$
see \eqref{eq:A_definition} and section \ref{sec:graph_metric}.  It is immediate that
$\rho_{G}$ is bounded from
below by $\delta= (M/\eta)^{-1/2}$.  The metric $\rho_{G}$ is \emph{adapted to the Laplacian} $\Delta_G$,
in view of the Definition (1.3) on p. 117 of \cite{Fo13}.  This means that for all $x\in VG$,
the metric $\rho_G$ satisfies the inequality
\begin{equation}\label{eq. intr ineq}
\frac{1}{\theta(x)}\sum_{y\sim x} \rho_G^2 (x,y)w_{xy}\leq 1
\end{equation}
and there exists a constant $c_{\rho_G}$ such that $\rho_G(x,y)\leq c_{\rho_G}$ whenever $x\sim y$. The inequality \eqref{eq. intr ineq}
is analogous to the geodesic distance $\rho$ on a Riemannian manifold which satisfies $|\nabla \rho(x, \cdot)|\leq 1$.
Any metric on $G$ adapted to the Laplacian is intrinsic.

Intrinsic metric for random walks under degenerate conductances, is defined, for $x, y \in VG$ by
$$
d_\theta^w(x,y):=\inf_{\gamma\in \Gamma_{xy}} \left\{\sum_{i=0}^{\ell_\gamma-1}  \min\left\{1,\frac{ \min\left\{\theta(z_i), \theta(z_{i+1})\right\}}{w_{z_iz_{i+1}}}\right\}^{1/2} \right\},
$$
where $ \Gamma_{xy}$ is the set of all nearest neighbor paths $\gamma=(z_0,\ldots, z_{\ell_{\gamma}})$ connecting $x$ and $y$, see \cite{ADS19} where it is also proved that $d_\theta^w$ satisfies \eqref{eq. metric bound from below}.


A further example of a metric on $G$ which is adopted to the Laplacian $\Delta_G$ and uniformly bounded from below is defined as follows.  For all $x, y \in VG$, let
$$
\tilde{d}_G(x,y):=\inf \left\{\sum_{e\in p(x,y)}  \min\{1, u(e)\} : p(x,y)\, \text{is a path joining }\, x\, \text{and} \, y\right\}
$$
where $e$ is an edge in the path and, if $x_i$ and $x_{i+1}$ are the endpoints of the edge $e$, $u(e)$ is defined as $$u(e)=\left(\min\left\{\frac{\theta(x_i)}{\mu(x_i)}, \frac{\theta(x_{i+1})}{\mu(x_{i+1})}\right\}\right)^{1/2}.$$

\subsection{Dilated Gaussian as a parametrix} \label{sec: ex of param}

Let $d: VG\times VG \to [0,\infty)$ denote any distance metric on $G$ satisfying the assumption \eqref{eq. metric bound from below}.

\begin{prop} \label{prop. Gaussian parametrix}
With the notation as above, let $H_d:VG\times VG \times (0,\infty) \to [0,\infty) $ be defined as
\begin{equation}\label{eq. defn param exp}
H_d(x,y;t):=\frac{1}{\sqrt{\theta(x)\theta(y)}} \exp(-(\theta(x)\theta(y)d^2(x,y))/t)
\end{equation}
and
$$
H_d(x,y;0)=\lim_{t\downarrow 0} H_d(x,y;t).
$$
Then $H_d$ is a parametrix for the heat operator on $G$ of order $k=0$.
\end{prop}

\begin{proof}
By assumption, the distance $d$ is bounded from below.  From this it is immediate
that $H_d(x,y;t)$ is continuous function in $t$ and that $H_d(x,y;0)$  satisfies
\eqref{eq:dirac property of heat kernel}. Trivially, $H_d(x,y;t)$ is smooth for $t\in(0,\infty)$ and all $x,y\in
VG$. By definition, $H_d$ is symmetric in the two spatial variables.  So, in order to prove integrability in each space
variable, it suffices to show that for all $y\in VG$ and $t>0$, we have that
\begin{equation*}
\sum_{x\in VG} \exp\left(-\frac{\theta(x)\theta(y)d^2(x,y)}{t}\right) \sqrt{\frac{\theta(x)}{\theta(y)}} 
\end{equation*}
First, we note that for all positive numbers $a$ and $u$, the inequality
$$
u\exp(-u^2a^2)\leq \frac{1}{a\sqrt{2e}}
$$
follows from elementary calculus.  When taking $u=\sqrt{ \theta(x)\theta(y)}d(x,y)$ for $x \neq y$ and $a=\frac{1}{\sqrt{2t}}$,
we get that
$$
 \exp\left(-\frac{\theta(x)\theta(y)d^2(x,y)}{2t}\right)\sqrt{\theta(x)}\leq \sqrt{\frac{t}{e}}\frac{1}{\sqrt{\theta(y)} d(x,y)}\leq \frac{1}{ \delta}\sqrt{\frac{t}{\eta e}},
$$
where we used (G2) and \eqref{eq. metric bound from below} to deduce the last inequality. Therefore,
$$
\frac{1}{\sqrt{\theta(y)}}\sum_{x\in VG} \exp\left(-\frac{\theta(x)\theta(y)d^2(x,y)}{t}\right) \sqrt{\theta(x)}\leq 1+
\sqrt{\frac{t}{ e}}\frac{1}{\delta\eta} \sum_{x\in VG\setminus\{y\}} \exp\left(-\frac{\theta(x)\theta(y)d^2(x,y)}{2t}\right).
$$
It is left to prove that the series on the right-hand side of the above equation converges. For this, note that
when combining (G3')  with \eqref{eq. metric bound from below} we conclude  that for fixed $y$ on $G$ there are at most $N^{n+1}$ vertices
$x\neq y$ with distance $\leq (n+1)\delta$ from $y$.  Therefore,
\begin{align}\label{eq. bound basic sum}
\sum_{x\in VG\setminus\{y\}} \exp\left(-\frac{\theta(x)\theta(y)d^2(x,y)}{2t}\right) &=  \sum_{n=1}^{\infty} \sum_{n\delta <d(x,y) \leq (n+1)\delta}  \exp\left(-\frac{\theta(x)\theta(y)d^2(x,y)}{2t}\right) \notag\\ &\leq
\sum_{n=1}^{\infty} N^{n+1}\exp\left(-\frac{(\eta n\delta)^2}{2t}\right)=C(N,\eta,\delta,t) <\infty.
\end{align}
This proves that
\begin{equation}\label{eq. L1 statement}
\frac{1}{\sqrt{\theta(y)}}\sum_{x\in VG} \exp\left(-\frac{\theta(x)\theta(y)d^2(x,y)}{t}\right) \sqrt{\theta(x)}\leq 1+\sqrt{\frac{t}{ e}}\frac{C(N,\eta,\delta,t)}{\delta\eta}:=C_1(N,\eta,\delta,t)
\end{equation}
and completes the proof that $H_d$ is integrable in both space variables.

Next, we want to show the second condition in Definition \ref{def. parametrix} of the parametrix holds for $H_{d}$. By definition, for $t>0$
we have that
\begin{align}\notag
L_{G,x}H_d(x,y;t) &= \frac{1}{\theta(x)}\sum_{z\sim x} \left( \frac{1}{\sqrt{\theta(x)\theta(y)}}
\exp\left(-\frac{\theta(x)\theta(y)d^2(x,y)}{t}\right) \right. \\& \hskip 1.0in
\left.- \frac{1}{\sqrt{\theta(z)\theta(y)}} \exp\left(-\frac{\theta(z)\theta(y)d^2(z,y)}{t}\right)\right)w_{xz}
\notag\\&+\frac{\theta(x)\theta(y)d^2(x,y)}{t^2 \sqrt{\theta(x)\theta(y)}}\exp\left(-\frac{\theta(x)\theta(y)d^2(x,y)}{t}\right).\label{eq. defn LGx}
\end{align}
The sum on the right-hand side \eqref{eq. defn LGx} is finite by (G3'), and each term extends to a continuous function in $t\in[0,\infty)$.
Therefore,  the second condition in Definition \ref{def. parametrix} is also fulfilled.

Going further, conditions (G1) and (G2) yield that
\begin{align*}
\left| \Delta_{G,x}H_d(x,y;t) \right|&\leq
\frac{1}{\theta(x)\eta} \sum_{z\sim x} \left( \exp\left(-\frac{\theta(x)\theta(y)d^2(x,y)}{t}\right) +  \exp\left(-\frac{\theta(z)\theta(y)d^2(z,y)}{t}\right)\right)w_{xz}\\ &\leq \frac{2\mu(x)}{\theta(x)\eta}\leq \frac{2M}{\eta}.
\end{align*}
For all $u, a > 0$, we have that
$$
u^2a^{-2}\exp(-u^2/a) \leq \frac{1}{ae},
$$
and for $u > 0$,
$$
\lim_{a\downarrow 0} u^2a^{-2}\exp(-u^2/a)=0.
$$
Upon taking $u= \sqrt{\theta(x)\theta(y)}d(x,y)$ and $a=t$, we get for $x\neq y$ and $t\in(0,t_0)$ that
\begin{align*}
\frac{\theta(x)\theta(y)d^2(x,y)}{t^2 \sqrt{\theta(x)\theta(y)}}&\exp\left(-\frac{\theta(x)\theta(y)d^2(x,y)}{t}\right) \\& \leq \frac{1}{\eta}\frac{\theta(x)\theta(y)d^2(x,y)}{t^2}\exp\left(-\frac{(\sqrt{\theta(x)\theta(y)}d(x,y))^2}{t}\right)\leq \frac{C}{\eta},
\end{align*}
where $C$ is a constant depending only on $t_0$.
When $x=y$, the third line in \eqref{eq. defn LGx} equals zero. With all this,  we have proved that $L_{G,x}H_d(x,y;t)$ is bounded for
$t\in(0,t_0)$ by a constant depending only on $t_0$.

Finally, in order to prove that $H_d$ is the parametrix of order zero, it is left to prove that
$\Delta_{G,x}H_d(x,\cdot;t)$ and  $\frac{\partial}{\partial t}  H_d(x,\cdot;t)$ are in  $L^1(\theta)$ for  $x\in VG$, $t>0$. We have
\begin{align*}\label{eq:LH_order_zero}
\sum_{y\in VG} |\Delta_{G,x}H_d(x,y;t)|\theta(y)&\leq \frac{1}{\theta(x)}\sum_{z\sim x}\frac{w_{xz}}{\sqrt{\theta(x)}} \sum_{y\in VG}\exp\left(-\frac{\theta(x)
\theta(y)d^2(x,y)}{t}\right)\sqrt{\theta(y)} \\&+ \frac{1}{\theta(x)}\sum_{z\sim x}\frac{w_{xz}}{\sqrt{\theta(z)}} \sum_{y\in VG}\exp\left(-\frac{\theta(z)\theta(y)d^2(z,y)^2}{t}\right)\sqrt{\theta(y)}.\notag
\end{align*}
In view of the symmetry of variables, from \eqref{eq. L1 statement} and (G1) we deduce that
$$
\sum_{y\in VG} |\Delta_{G,x}H_d(x,y;t)|\theta(y)\leq \frac{2C_1(N,\eta,\delta,t)}{\theta(x)}\sum_{z\sim x}w_{xz}\leq 2 M C_1(N,\eta,\delta,t).
$$
This proves that $\Delta_{G,x}H_d(x,\cdot;t)\in L^1(\theta)$.

To complete the proof of the proposition it is left to show that
$$
\sum_{y\in VG}\frac{\theta(x)\theta(y)d^2(x,y)}{t^2 \sqrt{\theta(x)\theta(y)}}\exp\left(-\frac{\theta(x)\theta(y)d^2(x,y)}{t}\right)\theta(y)
<\infty.
$$
The proof is completely analogous to the proof of \eqref{eq. L1 statement}, so we will omit further details.

\end{proof}

\section{Concluding remarks}

We will close the article with the following observations

\subsection{Choice of distance metric}
In Proposition \ref{prop. Gaussian parametrix} above, the choice of a distance metric was not specified.  Indeed, any distance metric $d$
which is uniformly bounded from below could be used in the definition \eqref{eq. defn param exp} for the parametrix. In other words, we
can construct the same heat kernel when using different distance metric on the graph.

This is quite different from the Riemannian manifolds
situation where the construction of a parametrix already requires considerable local information associated to the Laplacian;
see, for example, section VI.4 of \cite{Ch84}.  In a sense, the very general methods by which one can define a parametrix
for an infinite graph belongs to a class of geometric phenomena on the graph which are somewhat unexpected if one were
to view a graph as the discretization of a manifold. For an extensive discussion of such phenomena we refer to \cite{KLW21}.

\subsection{Edge-weighted graph distance}\label{sec. edge weighted dist}

A natural condition, which arises in many studies related to Gaussian-type bounds for the heat kernel on infinite graphs (e.g. \cite{HLLY19}, \cite{Wu21}) is

\begin{itemize}
\item [(E1)] {\bf Boundedness from below of the edge weight.}
There exists a positive number $\tilde{\delta}$ such that
$$
\inf\limits_{x,y\in VG; x\neq y} w_{xy} >\tilde{\delta}.
$$
\end{itemize}


Assumption (E1) arises naturally in the context of metric graphs in which a positive real number $\ell_e$ is
associated to every edge $e=\{x,y\}$ of a graph, so then $w_{xy}$ can be taken be equal $\ell_e$.  The condition
$\inf_e \ell_e >0$ is then imposed in order to deduce that a connected metric graph is actually a length
metric space; see for example \cite{BBI01} or \cite{St06}.

With (E1) another distance on $G$ can be defined as follows.  For any vertices $x$ and $y$ and any
path $p(x,y)=\{x=x_0, x_1,\ldots,
x_n=y\}$ connecting points $x$ and $y$, the \emph{weighted length} of the path $\ell_w(p(x,y))$
is defined by
$$
\ell_w(p(x,y)) :=\sum_{j\geq 0}w_{x_jx_{j+1}}.
$$
Let
$$
\tilde{d}_G(x,y):= \inf_{p(x,y)} \ell_w(p(x,y))
$$
where the infimum is taken over all paths connecting $x$ and $y$.
Since $G$ is connected and $w_{xy}$ is uniformly bounded from below,
the function $\tilde{d}_G: VG \times VG  \to [0,\infty]$ is a distance on $VG$ which can be called the
\emph{edge-weighted distance}. For graphs satisfying (G1), (G2), (G3') and (E1) it is immediate that
the dilated Gaussian $H_{\tilde{d}_G}$ defined by  \eqref{eq. defn param exp} is a parametrix.

\subsection{Volume doubling}

The condition (G3') was used in the proof of Proposition \ref{prop. Gaussian parametrix} in order to derive that the right-hand side of \eqref{eq. defn LGx} extends to a continuous function in $t$ and in order to derive the bound \eqref{eq. bound basic sum}.

The finiteness of the sum on the right-hand side of \eqref{eq. defn LGx} follows from local finiteness of $G$, meaning that it holds true even if the vertex degree is not uniformly bounded. The bound \eqref{eq. bound basic sum} can be proved if, instead of (G3') one poses the following condition on the volume growth.

\begin{itemize}
\item [(V1)] {\bf Volume doubling for distance $d$.}
There exists a positive constant $C_d$ such that for all $x\in VG$ and all $r>0$
$$
V_d(x,2r)\leq C_d V_d(x,r).
$$
\end{itemize}
Here, the volume $V_d(x,r)$ of a ball centered at $x$ of radius $r$, in metric $d$ on a graph $G$ is defined as
$$
 V_d(x,r)=\sum_{y\in VG\,:\, d(x,y)<r } \theta(y).
$$
If metric $d$ satisfies \eqref{eq. metric bound from below} and $G$ has volume doubling property, it is trivial to deduce that
$$
V_d(x,2^n \delta) \leq C_d^n V(x,\delta) =C_d^n \theta(x).
$$
Then
\begin{align*}
\sum_{x\in VG\setminus\{y\}} \exp\left(-\frac{\theta(x)\theta(y)d^2(x,y)}{2t}\right) &=  \sum_{n=1}^{\infty} \sum_{2^n\delta <d(y,x) \leq 2^{n+1}\delta}  \frac{\theta(x)}{\eta}\exp\left(-\frac{\theta(x)\theta(y)d^2(y,x)}{2t}\right) \notag\\ &\leq
\frac {1}{\eta}\sum_{n=1}^{\infty} C_d^{n+1}\theta(y)\exp\left(-\frac{\eta \theta(y) (2^n\delta)^2}{2t}\right)<\infty.
\end{align*}
This proves that \eqref{eq. bound basic sum} holds true when (G3') is replaced by (V1) and shows that Proposition \ref{prop. Gaussian parametrix} holds true for locally finite graphs $G$  satisfying conditions (G1), (G2) and with metric $d$ such that \eqref{eq. metric bound from below} and (V1) hold true.

\vspace{5mm}
\noindent
Jay Jorgenson \\
 Department of Mathematics \\
 The City College of New York \\
 Convent Avenue at 138th Street \\
 New York, NY 10031 U.S.A. \\
 e-mail: jjorgenson@mindspring.com

\vspace{5mm}
\noindent
Anders Karlsson \\
 Section de mathématiques\\
 Université de Genève\\
 2-4 Rue du Liévre\\
 Case Postale 64, 1211\\
 Genève 4, Suisse\\
 e-mail: anders.karlsson@unige.ch \\
 and \\
 Matematiska institutionen \\
 Uppsala universitet \\
Box 256, 751 05 \\
 Uppsala, Sweden \\
 e-mail: anders.karlsson@math.uu.se

\vspace{5mm}
\noindent
Lejla Smajlovi\'{c} \\
 Department of Mathematics and Computer Science\\
 University of Sarajevo\\
 Zmaja od Bosne 35, 71 000 Sarajevo\\
 Bosnia and Herzegovina\\
 e-mail: lejlas@pmf.unsa.ba
\end{document}